\documentclass[11pt]{article}
\usepackage{latexsym,amsfonts,amssymb,amscd,amsmath,amsthm}

\usepackage[width=16cm, height=22cm, top=2cm, bottom=2cm]{geometry}

\usepackage{leftidx}
\usepackage[latin1]{inputenc}%
\usepackage{amssymb}%
\usepackage{amscd}%
\usepackage{graphicx}%
\usepackage{amsthm}
\usepackage{amsmath}
\usepackage{amsfonts} 
\usepackage{latexsym}%
\usepackage{epsfig}%
\usepackage{epic}%
\usepackage{pstricks}%
\usepackage{pst-node}%
\usepackage{pst-text}%
\usepackage{pst-3d}%

\theoremstyle{plain}
\newtheorem{thm}{Theorem}[section]
\newtheorem{cor}[thm]{Corollary}
\newtheorem{lem}[thm]{Lemma}
\newtheorem{prop}[thm]{Proposition}
\theoremstyle{definition}

\newtheorem{exam}[thm]{Example}
\newtheorem{defn}[thm]{Definition}

\theoremstyle{remark}
\newtheorem{rem}[thm]{Remark}
\numberwithin{equation}{section}

\begin{document}

\title{Measure Algebras on Homogeneous Spaces}
\author{Hossein Javanshiri and Narguess Tavallaei
}
\maketitle

\begin{abstract}

For a locally compact group $G$ and a compact subgroup $H$, we show that the Banach space $M(G/H)$ may be considered as a quotient space of $M(G)$. Also, we define a convolution on $M(G/H)$ which makes it into a Banach algebra. It may be identified with a closed subalgebra of the involutive Banach algebra $M(G)$, and there is no involution on $M(G/H)$ compatible with this identification unless $H$ is a normal subgroup of $G$. In other words, $M(G/H)$ is a $*$-Banach subalgebra of $M(G)$ only if $H$ is a normal subgroup of $G$. As well, it is a unital Banach algebra just when $H$ is a normal subgroup. Furthermore, when $G/H$ is attached to a strongly quasi-invariant measure, $L^1(G/H)$ is a Banach subspace of $M(G/H)$. Using the restriction of the convolution on $M(G/H)$, we obtain a Banach algebra $L^1(G/H)$, which may be considered as a Banach subalgebra of $L^1(G)$, with a right approximate identity. It has no involution and no left approximate identity except for a normal subgroup $H$. Consequently, the Banach algebra $L^1(G/H)$ is amenable if and only if $H$ is a normal subgroup and $G$ is amenable.

\end{abstract}

\section{Preliminaries}

For a locally compact Hausdorff space $X$, we denote the
space of all continuous complex-valued functions on
$X$ that vanishes at infinity by $C_0 (X)$, the subspace of $C_0 (X)$ consists of all functions with compact supports by $C_c (X)$, and the Banach space of all complex regular Borel measure on $X$ by $M(X)$. When $\mu$ is a Radon measure on $X$,
for all $1\leq p <+\infty$, $(L^p(X,\mu),\|.\|_p)$ stands for the
Banach space of the equivalence classes of all $\mu\,$-measurable
complex-valued functions on $X$ whose $p\,$th powers are
integrable with
\[\|f\|_p=(\int_X |f(x)|^p\,d\mu (x))^{1/p}\quad(f\in L^p(X,\mu)).\]
Also, by $(L^\infty(X,\mu),\|.\|_\infty)$ we mean the Banach space
of all equivalence classes of locally essentially bounded $\mu\,$-measurable
complex-valued functions on $X$ with the usual norm $\|.\|_\infty$. When $\nu$ is a Radon measure on $X$ and $f\in L^1(X,\nu)$, we mean by $\nu_f$ the Radon measure on $X$ satisfying $d\nu_f(x)=f(x)\,d\nu(x)$.\\

Throughout, we suppose that $G$ is a locally
compact group with neutral element $e$, left Haar measure $\lambda_G$, and
modular function $\Delta_G$. Also, $H$ is a closed subgroup of $G$
with left Haar measure $\lambda_H$ and modular function $\Delta_H$ and $q:G\longrightarrow G/H$ is the canonical
projection. The
quotient space $G/H$ is considered as a homogeneous
space that $G$ acts on it from the left, by
$x(yH)=(xy)H$. \\
When $\mu$ is a Radon measure on $G/H$, for all $x\in G$, the translation $\mu_x$ of $\mu$ is defined by $\mu_x (E)=\mu (xE)$,
where $E$ is a Borel subset of $G/H$. A measure $\mu$ is said to
be {\it $G$-invariant} if $\mu_x =\mu$ for all
$x\in G$, and is said to be {\it strongly quasi--invariant}
provided that a continuous function $\lambda :G\times G/H
\longrightarrow (0,+\infty)$ exists which satisfies
\begin{eqnarray*}\label{rho6}
d\mu_x (yH)=\lambda (x,yH)\,d\mu (yH)\quad(x,y\in G).
\end{eqnarray*}
Also, a rho-function for the pair $(G,H)$ is a continuous function
$\rho:G\longrightarrow (0,+\infty)$ with
\[\rho(x\xi)=\frac{\Delta _{H}(\xi)}{\Delta _{G}(\xi)}\,\rho(x)\quad(x\in G,\;\xi \in H).\]
We should mention that $(G,H)$ always admits a rho-function and for a given rho-function $\rho$, there exists a strongly quasi--invariant measure $\mu$ on $G/H$ such that the Weil's formula holds:
\begin{eqnarray}\label{rho2}
\int_G f(x)\,\rho(x)\,d\lambda_G(x) = \int_{G/H}\int_H f(x\xi)\,d\lambda_H(\xi)\,d\mu(xH)\quad(f\in C_c(G)).
\end{eqnarray}
In addition, all strongly quasi--invariant measures on $G/H$ arise
from rho-functions as \eqref{rho2}, and all
these measures are strongly equivalent. Specially, if $\mu$ is a strongly quasi--invariant measure, then so are its translations and
\begin{eqnarray*}\label{rho6}
\frac{d\mu_x}{d\mu}(yH) &=& \frac{\rho(xy)}{\rho(y)}\quad(x,y\in
G).
\end{eqnarray*}
(For details see \cite[Subsection~2.6]{farmonic}).\\


Moreover, if $\mu$ is the strongly quasi--invariant
measure on $G/H$ which arises from a rho-function $\rho$ and $L^1(G/H)=L^1(G/H,\mu)$, then the mapping
$T_\rho:L^1(G)\mapsto L^1(G/H)$ defined by
\begin{eqnarray}\label{tif}
T_\rho f(xH)=\int_H
\frac{f(x\xi)}{\rho(x\xi)}\,d\lambda_H(\xi)\quad(\mu -\text {almost all }~xH\in
\frac{G}{H})
\end{eqnarray}
is a surjective bounded linear map with $\|T_\rho \|\leq 1$ and
\begin{eqnarray*}
\int_{G/H}T_\rho f(xH)\,d\mu(xH)=\int_G f(x)\,d\lambda_G(x)\quad(f\in L^1(G)).
\end{eqnarray*}
More
precisely, for all $\varphi\in L^1(G/H)$ there exists
some $f\in L^1(G)$ such that $\varphi=T_\rho f$ and such that $f\geq 0$
almost everywhere if $\varphi\geq 0$ $\mu\,$-almost everywhere. Also,
\begin{eqnarray}\label{norminf} \|\varphi\|_1=\inf\{\|f\|_1:\,f\in
L^1(G),\,\varphi=T_\rho f\}
\end{eqnarray}
for all $\varphi\in L^1(G/H)$ and
\begin{eqnarray}
\|\varphi\|_1=\inf\{\|f\|_1:\,f\in C_c(G),\,\varphi=T_\rho f\},
\end{eqnarray}
where $\varphi\in C_c(G/H)$. It follows that the Banach space $L^1(G/H)$ is isometrically isomorphic to the quotient space $L^1(G)/Ker(T_\rho)$ equipped with the usual quotient norm (For more details see \cite[Subsection~3.4]{reiter}). Obviously, the function $T_\rho$ is dependent on the choice $\rho$ and when there is no ambition we use notation $T$ instead of $T_\rho$.\\

When $H$ is a closed subgroup
of $G$ and $\mu$ is a strongly quasi-invariant Radon measure on $G/H$, we may define the left translation of some $\varphi\in L^1(G/H)$, by an element $x$ of $G$, by $\mathcal{L}_x\varphi=T(L_x f)$, where $f\in L^1(G)$ and $\varphi = Tf$. In other words, we have
\begin{eqnarray}\label{ltrn}
\mathcal{L}_x\varphi(yH)=\lambda(x^{-1},yH)\,\varphi(x^{-1}yH)\quad(\mu -\text {almost all }~yH\in
G/H)
\end{eqnarray}
Also, for all $\varphi\in L^\infty(G/H)$ we may define
\begin{eqnarray}
\mathcal{L}_x\varphi(yH)=\varphi(x^{-1}yH)\quad(\mu -\text {locally almost all }~yH\in
G/H)
\end{eqnarray}

A function $\varphi$ on $G/H$ is said to be left uniformly continuous if for all
$\varepsilon >0$ there exists a neighborhood $U$ of $e$ such that for all $x\in U$ and $y\in G$, we have $|\varphi(xyH)-\varphi(yH)|<\varepsilon$.
We denote the set of all bounded left uniformly
continuous functions of $G/H$ by $LUC(G/H,G)$. Therefore, we get $C_0(G/H)\subseteq LUC(G/H,G)$.
Moreover, the homomorphism $G\rightarrow B(L^1(G))$, $x\mapsto \mathcal{L}_x$, is continuous with respect to the strong operator topology on $B(L^1(G/H))$. Also, for all $\varphi\in L^\infty(G/H)$, the mapping $x\mapsto \mathcal{L}_x\varphi$
is continuous only if $\varphi$ is left uniformly continuous.

Now, suppose that $H$ is a compact subgroup of $G$. It is easy to check that for all\break $\varphi\in C_0(G/H)$, $\varphi oq\in C_0(G)$ and it is constant on the left cosets of $H$. On the other hand, if\break $f\in C_0(G)$, then $T_\infty f\in C_0(G/H)$ where $T_\infty f(xH)=\int_H f(x\xi)\,d\lambda_H(\xi)$, $x\in G$. Moreover,\break
$T_\infty:C_0(G)\rightarrow C_0(G/H)$, $f\mapsto T_\infty f$ is a norm decreasing linear map between Banach spaces $C_0(G)$ and $C_0(G/H)$. Assume that
\begin{eqnarray*}
C_0(G:H)=\{f\in C_0(G):\,f(xh)=f(x), x\in G,\,h\in H\}.
\end{eqnarray*}
One may easily check that $C_0(G:H)$ is a Banach subalgebra of $(C_0(G), \|.\|_{\sup})$ and the restriction of $T_\infty$ to $C_0(G:H)$ determines an isometry isomorphism between Banach spaces $C_0(G:H)$ and $C_0(G/H)$. Also, for all $\varphi\in L^1(G/H)$ and $x\in G$ we mean by the right translation of $\varphi$ by $x$ in $L^1(G/H)$, the element $\mathcal{R}_x\varphi$ of $L^1(G/H)$ which is determined by $\mathcal{R}_x \varphi =T(R_x \varphi_\rho)$. Likewise, for all $\varphi\in L^\infty(G/H)$ and $x\in G$ we define $\mathcal{R}_x\varphi\in L^\infty(G/H)$ by
\begin{eqnarray}
\mathcal{R}_x \varphi(yH) =\int_H \varphi(y\xi xH)\,d\lambda_H(\xi)\quad(\mu -\text {locally almost all }~yH\in
G/H).
\end{eqnarray}
Then $\mathcal{R}_x:L^1(G/H)\rightarrow L^1(G/H)$, $\varphi\mapsto \mathcal{R}_x\varphi$ is a bounded linear operator whose adjoint is $\Delta_G(x)^{-1}\,\mathcal{R}_{x^{-1}}:L^\infty(G/H)\rightarrow L^\infty(G/H)$. Moreover, the mapping $G\rightarrow B(L^1(G))$, $x\mapsto \mathcal{R}_x$ is a continuous homomorphism with respect to the strong operator topology on $B(L^1(G/H))$. But the mapping $\mathcal{R}:G\rightarrow L^\infty(G)$, $x\mapsto \mathcal{R}_x\varphi$ is not necessarily continuous where $\varphi\in L^\infty(G/H)$. We mean by a bounded right uniformly continuous function on $G/H$, a bounded continuous function for which the mapping $G\rightarrow L^\infty(G/H)$, $x\mapsto \mathcal{R}_x\varphi$ is continuous. We denote the set of all bounded right uniformly continuous function $G/H$ by $RUC(G/H,G)$. Then $RUC(G/H,G)$ is a Banach subspace of $C_b(G/H)$, with resect to $\|.\|_{\sup}$, containing $C_0(G/H)$.\\


\section{$M(G/H)$ as a Banach Left $M(G)$-Module}

Reiter and Stegeman have remarked in \cite{reiter} that for all closed subgroups $H$ of $G$, one may define a norm decreasing linear map as follow:
\begin{eqnarray*}
\left\{\begin{array}{clcr}
\tilde{T}:M(G)\rightarrow M(G/H)\\
\tilde{T}m(E)=m(q^{-1}(E))
\end{array}
\right.
\end{eqnarray*}
for all Borel subsets $E$ of $G/H$ and $m\in M(G)$. In other words,
\begin{eqnarray}
\int_{G/H}\varphi(xH)\,d\tilde{T}m(xH)=\int_G \varphi(xH)\,dm(x)\quad(\varphi\in C_0(G/H))
\end{eqnarray}
Becides, when $\mu$ is a strongly quasi-invariant measure on $G/H$, $\tilde{T}$ is an extension of\break $T:L^1(G)\rightarrow L^1(G/H,\mu)$ defined by \eqref{tif} (For details see \cite[Subsection~3.4]{reiter}). Therefore, by using the kernel of $\tilde{T}$, one may consider $M(G/H)$ as a quotient of the Banach space $M(G)$. \\

\begin{exam}\label{exam01}
Let $G$ be a locally compact group.
\newcounter{mh3}
\begin{list}
{\bf(\alph{mh3})}{\usecounter{mh3}}
\item We know that $\lambda_G\in M(G)$ just when $G$ is compact. In this case, every closed subgroup $H$ of $G$ is compact and has a normalized Haar measure $\lambda_H$. If $\mu$ is the invariant measure on $G/H$ arising from the constant rho-function $\rho=1$, then $\tilde{T}\lambda_G =\mu$.
\item For all $x\in G$, $\tilde{T}(\delta_x)=\delta_{xH}$.
\item Let $\mu$ be a strongly quasi--invariant measure on $G/H$. For all $f\in L^1(G)$, we have $\tilde{T}{\lambda}_f=\mu_{Tf}$.
\end{list}
\end{exam}

\begin{rem}
Part {\bf (c)} of Example \ref{exam01} states that when $m$ is absolutely continuous with respect to the left Haar measure $\lambda$ on $G$ and $dm(x)=f(x)\,d\lambda_G(x)$, $f\in L^1(G)$, then $\tilde{T}m$ is absolutely continuous with respect to $\mu$ and $d\tilde{T}m(xH)=Tf(xH)\,d\mu(xH)$. Generally, we can say the transform $\tilde{T}:M(G)\rightarrow M(G/H)$ preserves absolute continuity, i.e.; if $m_1,m_2\in M(G)$ are such that $m_1\ll m_2$, then $\tilde{T}m_1\ll\tilde{T}m_2$. Since if $E$ is a Borel subset of $G/H$ and $\tilde{T}m_2(E)=0$, then $m_2(q^{-1}(E))=\tilde{T}m_2(E)=0$ and so $\tilde{T}m_1(E)=m_1(q^{-1}(E))=0$.

\end{rem}

Using the left action of $G$ on $G/H$, we may define a left action of $M(G)$ on $M(G/H)$ as follows:\\
For all $m\in M(G)$ and $\nu\in M(G/H)$ define $m \star\nu\in M(G/H)$ by
\begin{eqnarray}
m \star\nu(\varphi)=\int_{G/H} \int_G \varphi(yxH)\,dm(y)\,d\nu(xH)\quad(\varphi\in C_0(\frac{G}{H}))
\end{eqnarray}
This action makes $M(G/H)$ into a Banach left $M(G)$-module.\\

From another point of view, this action obtained by a transferring of the convolution of $M(G)$ via $\tilde{T}$. In details, for all $m\in M(G)$ and $\nu\in M(G/H)$, $m \star\nu=\tilde{T}(m*m_0)$, where $m_0\in M(G)$ is such that $\nu=\tilde{T}(m_0)$. To show this, we need the following proposition and the obvious fact that $\tilde{T}(m*m_0)$ is independent of the choice $m_0$ with $\tilde{T}(m_0)=\nu$.

\begin{prop}
For all $m_1,m_2\in M(G)$, $\tilde{T}(m_1*m_2)=m_1\star\tilde{T}m_2$.
\end{prop}

\begin{proof}
For all $m_1,m_2\in M(G)$ and $\varphi\in C_0(G/H)$,
\begin{eqnarray*}
\int_{G/H}\varphi(xH)\,d\tilde{T}(m_1*m_2)(xH) &=& \int_{G}\varphi(xH)\,d(m_1*m_2)(x) \\
&=& \int_{G}\int_G \varphi(xyH)\,dm_1(x)\,dm_2(y) \\
&=& \int_{G/H} \int_{G} \varphi(xyH)\,dm_1(x)\,d(\tilde{T}m_2)(yH) \\
&=& \int_{G/H}\varphi(xH)\,d(m_1\star\tilde{T}m_2)(xH),
\end{eqnarray*}
for $\mu$-almost all $yH\in G/H$.
\end{proof}

\section{$M(G/H)$ as a Closed Subspace of $M(G)$}

From now on, we suppose that $H$ is a compact subgroup of $G$. We aim to prove that, in this case, $M(G/H)$ is isometrically isomorphic to a specific closed subspace of $M(G)$. Besides, it is identified with the dual of a closed subalgebra of $C_0(G)$.\\

Recall that by taking some $x\in G$ and $f\in C_0(G)$, we obtain two functions $f_x,\,\leftidx{_x}f\in C_0(G)$, defined by $f_x(y)=f(yx)$ and $\leftidx{_{x}}f(y)=f(xy)$, and the mappings follow are separately continuous:
\begin{eqnarray*}
\left\{\begin{array}{clcr}
G\times C_0(G)\rightarrow C_0(G)\\
(x,f)\mapsto f_x
\end{array}
\right.
\quad\text{ and}\quad
\left\{\begin{array}{clcr}
G\times C_0(G)\rightarrow C_0(G)\\
(x,f)\mapsto \leftidx{_x} f
\end{array}
\right.
\end{eqnarray*}
where $C_0(G)$ has been attached to $\|.\|_{\sup}$. We first show that for all $x\in G$ and $\varphi\in C_0(G/h)$, there are two functions $\varphi_{xH},\,\leftidx{_{xH}}\varphi\in C_0(G)$ with the similar properties.

\begin{lem}\label{lma02}
For all $\varphi\in C_0(G/H)$ and $x\in G$, the functions defined as follow belong to $C_0(G/H)$:
\begin{eqnarray*}
\left\{\begin{array}{clcr}
\leftidx{_{xH}}\varphi:G/H\rightarrow \mathbb{C}\\
yH\mapsto\int_H \varphi(x\xi yH)\,d\lambda_H(\xi)
\end{array}
\right.
\quad\text{ and}\quad
\left\{\begin{array}{clcr}
\varphi_{xH}:G/H\rightarrow \mathbb{C}\\
yH\mapsto\mathcal{R}_x\varphi(yH)=\int_H \varphi(y\xi xH)\,d\lambda_H(\xi)
\end{array}
\right.
\end{eqnarray*}
\end{lem}

\begin{proof}
As $\varphi oq\in C_0(G)$ is uniformly continuous, for given $\varepsilon>0$ there exists a symmetric neighborhood $U$ of $e$ such that $|\varphi o q (y)-\varphi oq(z)|<\varepsilon$, where $y,z\in G$ and $yz^{-1}\in U$ or $y^{-1}z\in U$.
Now, if $y,z\in G$ with $y^{-1}z\in U$, then for all $\xi\in H$ we have $(x\xi y)^{-1}(x\xi z)=y^{-1}z\in U$
and hence
\begin{eqnarray*}
|\leftidx{_{xH}}\varphi(yH)-\leftidx{_{xH}}\varphi(zH)| \leq \int_H|(\varphi oq)(x\xi y)-(\varphi oq)(x\xi z)|\,d\lambda_H(\xi)<\varepsilon.
\end{eqnarray*}
If $zy^{-1}\in U$, then $|\varphi_{xH}(yH)-\varphi_{xH}(zH)|<\varepsilon$. Therefore, the mappings $y\mapsto\int_H \varphi(x\xi yH)\,d\lambda_H(\xi)$ and $y\mapsto\int_H \varphi(y\xi xH)\,d\lambda_H(\xi)$ are continuous. This implies that the mappings\break $\leftidx{_{xH}}\varphi:\,yH\mapsto\int_H \varphi(x\xi yH)\,d\lambda_H(\xi)$ and $\varphi_{xH}:\,yH\mapsto\int_H \varphi(y\xi xH)\,d\lambda_H(\xi)$ are continuous.\\
Moreover, for given $\varepsilon>0$, take a compact subset $K$ of $G$ for which $|\varphi(zH)|<\varepsilon$ if $zH\notin q(K)$. Then trivially, $Hx^{-1}K$ is a compact subset of $G$ and if $yH\notin q(Hx^{-1}K)$, then $x\xi yH\notin q(K)$ and hence $|\varphi(x\xi yH)|<\varepsilon$ for all $\xi\in H$. This shows that $\leftidx{_{xH}}\varphi\in C_0(G/H)$. Similarly, $\varphi_{xH}\in C_0(G/H)$.
\end{proof}

\begin{defn}
For all $x\in G$ define
\begin{eqnarray*}
\left\{\begin{array}{clcr}
LT_{xH}:C_0(G/H)\rightarrow C_0(G/H)\\
\varphi\mapsto \leftidx{_{xH}}\varphi
\end{array}
\right.
\quad\text{ and}\quad
\left\{\begin{array}{clcr}
RT_{xH}:C_0(G/H)\rightarrow C_0(G/H)\\
\varphi\mapsto \varphi_{xH}
\end{array}
\right.
\end{eqnarray*}

\end{defn}

For all $\nu\in M(G/H)$ define
\begin{eqnarray*}
\left\{\begin{array}{clcr}
S_\nu:C_0(G) \rightarrow \mathbb{C}\\
f\mapsto \int_{G/H} T_\infty f(xH)\,d\nu(xH)
\end{array}
\right.
\end{eqnarray*}
Clearly, $S_\nu$ is a bounded linear functional on $C_0(G)$ with $\|S\|\leq 1$. Therefore, we can assign to each $\nu\in M(G/H)$ an element $m_\nu\in M(G)$ such that
\begin{eqnarray}
\int_G f(x)\,dm_{\nu}(x)=\int_{G/H} T_\infty f(xH)\,d\nu(xH)= \int_{G/H} \int_H f(x\xi)\,d\lambda_H(\xi)\,d\nu(xH)
\end{eqnarray}
for all $f\in C_0(G)$. We should note that if for a given $\nu\in M(G/H)$ we take the element $m_\nu\in M(G)$ and then return to $M(G/H)$ by using $\tilde{T}$, we obtain the same measure $\nu$, i.e.;
\begin{eqnarray}\label{mh01}
\tilde{T}(m_{\nu})=\nu\quad(\nu\in M(G/H)).
\end{eqnarray}
Equation \eqref{mh01} shows that $\|\nu\|=\|\tilde{T}(m_{\nu})\|\leq \|m_{\nu}\|\leq \|\nu\|$. Consequently,
\begin{eqnarray}\label{mh02}
\|\nu\|=\|m_\nu\|\quad(\nu\in M(G/H))
\end{eqnarray}
and
\begin{eqnarray}
\|\nu\|=\min\{\|m\|:\,m\in M(G),\,\nu=\tilde{T}(m)\}\quad(\nu\in M(G/H)).
\end{eqnarray}

\begin{lem}\label{lma01}
For a measure $m\in M(G)$, the following are equivalent:
\newcounter{mh1}
\begin{list}
{\bf(\alph{mh1})}{\usecounter{mh1}}
\item $m=m_{\tilde{T}m}$.
\item $m=m_\nu$ for some $\nu\in M(G/H)$.
\item $m(Ah)=m(A)$, where $A$ is a Borel subset of $G$ and $h\in H$.
\item $m(f_h)=m(f)$, where $f\in C_0(G)$ and $h\in H$.
\end{list}
\end{lem}

\begin{proof}
{\bf (b $\Rightarrow$ a)} If $m=m_\nu$ for some $\nu\in M(G/H)$, then $\tilde{T}m=\tilde{T}(m_\nu)=\nu$ and so $m=m_{\tilde{T}m}$.\\
{\bf (a $\Rightarrow$ c)} For all Borel subset $A$ of $G$ and $h\in H$ we have
\begin{eqnarray*}
m(Ah) &=& m_{\tilde{T}m} (Ah)\\
&=& \int_G \chi_{Ah}(x)\,dm_{\tilde{T}m}(x)\\
&=& \int_{G/H} \int_H \chi_{Ah}(x\xi)\,d\lambda_H(\xi)\,d{\tilde{T}m}(xH)\\
&=& \int_{G/H} \int_H \chi_{A}(x\xi)\,d\lambda_H(\xi)\,d{\tilde{T}m}(xH)\\
&=& m_{\tilde{T}m} (A)\\
&=& m(A),
\end{eqnarray*}
{\bf (c $\Rightarrow$ a)}
Suppose that $m\in M(G)$ satisfies (c). Then for all Borel subset $A$ of $G$ we can write
\begin{eqnarray*}
m_{\tilde{T}m} (A) &=& \int_G \chi_{A}(x)\,dm_{\tilde{T}m}(x)\\
&=& \int_{G/H} \int_H \chi_{A}(x\xi)\,d\lambda_H(\xi)\,d{\tilde{T}m}(xH)\\
&=& \int_{G} \int_H \chi_{A}(x\xi)\,d\lambda_H(\xi)\,dm(xH)\\
&=& \int_H  \int_{G} \chi_{A}(x\xi)\,dm(xH)\,d\lambda_H(\xi)\\
&=& \int_H  m(A\xi^{-1})\,d\lambda_H(\xi)\\
&=& m(A),
\end{eqnarray*}
which shows that $m=m_{\tilde{T}m}$.
\end{proof}

We denote the set of all measures $m\in M(G/H)$ satisfying equivalent conditions in Lemma~\ref{lma01} by $M(G:H)$. That is
\begin{eqnarray}
M(G:H)&=&\{m_\nu:\,\nu\in M(G/H)\}\nonumber\\
&=&\{m:\,m(Ah)=m(A),\,A\text{ is a Borel subset of }G,\,h\in H\}.
\end{eqnarray}

\begin{prop}
$M(G:H)$ is a Banach subalgebra and also a left ideal of $M(G)$.
\end{prop}

\begin{proof}
It is easy to see that $M(G:H)$ is a subspace of $M(G)$. Now, if $m\in M(G:H)$ and $n\in M(G)$, then
\begin{eqnarray*}
(n*m)(Ah)&=& \int_G m(x^{-1}Ah)\,dn(x)\\
&=& \int_G m(x^{-1}A)\,dn(x)\\
&=&(n*m)(A),
\end{eqnarray*}
where $A$ is a Bore subset of $G$ and $h\in H$. Hence, $M(G:H)$ is a left ideal of $M(G)$. Also, if $\{m_\alpha\}_{\alpha\in I}\subseteq M(G:H)$ tends to some $m\in M(G)$, then for all $f\in C_0(G)$ and $h\in H$ we have
\begin{eqnarray*}
m(f_h) = \lim_\alpha m_\alpha (f_h) = \lim_\alpha m_\alpha (f) = m(f),
\end{eqnarray*}
which shows that $m\in M(G:H)$.
\end{proof}

Equations \eqref{mh01} and \eqref{mh02} show that the restriction
\begin{eqnarray*}
\left\{\begin{array}{clcr}
M(G:H)\rightarrow M(G/H)\\
m\rightarrow \tilde{T}m
\end{array}
\right.
\end{eqnarray*}
of $\tilde{T}$ is an isometry isomorphism of Banach spaces. Therefore, $M(G/H)$ has been identified with the closed subspace $M(G:H)$ of $M(G)$. \\

It is worthwhile to mention that $M(G:H)\cong C_0(G:H)^*$. In details, by considering the canonical isomorphisms $M(G)\rightarrow C_0(G)^*$ and $M(G/H)\rightarrow C_0(G/H)^*$, and the isomorphism $T_\infty|_{C_0(G:H)}:C_0(G:H)\rightarrow C_0(G/H)$, the mapping
\begin{eqnarray*}
\left\{\begin{array}{clcr}
U:M(G:H)\rightarrow C_0(G:H)^*\\
m_\nu\mapsto m_\nu|_{C_0(G:H)}
\end{array}
\right.
\end{eqnarray*}
is an isometry isomorphism and the following diagram commutes:
\[
  \begin{CD}
   M(G:H) @> U >> C_0(G:H)^*\\
    @ V  VV  @ AAA\\
    M(G/H) @>> > C_0(G/H)^*
  \end{CD}
  \]

\bigskip

\begin{prop}
An element $f\in L^1(G)$ belongs to $L^1(G:H)$ if and only if $\lambda_f\in M(G:H)$, i.e.;
\begin{eqnarray}
L^1(G:H) = \{f\in L^1(G):\,\lambda_f\in M(G:H)\}.
\end{eqnarray}
\end{prop}

\begin{proof}
For all $f\in L^1(G)$, we have
\begin{eqnarray*}
\int_A (R_hf-f)(x)\,d\lambda_G(x) &=& \int_G \chi_{A}(x)\,(R_hf(x)-f(x))\,d\lambda_G(x)\\
&=& \int_G \chi_{A}(x)\,f(xh)\,d\lambda_G(x) -\int_G \chi_{A}(x)\,f(x)\,d\lambda_G(x)\\
&=& \int_G \chi_{Ah^{-1}}(x)\,f(x)\,d\lambda_G(x) -\int_G \chi_{A}(x)\,f(x)\,d\lambda_G(x)\\
&=& \lambda_f(Ah^{-1})-\lambda_f(A),
\end{eqnarray*}
where $h\in H$ and $A$ is a Borel subset of $G$. This shows that $\lambda_f \in M(G:H)$ just when $R_hf=f$ in $L^1(G)$ for all $h\in H$.
\end{proof}

\begin{exam}
When $\varphi\in L^1(G/H)$ we have $\varphi_\rho=\varphi oq\,.\,\rho\in L^1(G:H)$ (cf. \cite{tro}). So $\lambda_{\varphi_\rho}$ must be in $M(G:H)$. In fact, $\lambda_{\varphi_\rho}=m_{\mu_\varphi}$. Because,
\begin{eqnarray*}
\lambda_{\varphi_\rho}(g) &=& \int_G \varphi_\rho(x)\,g(x)\,d\lambda_G(x)\\
&=& \int_{G/H} T(\varphi oq\,\rho\,g)(xH)\,d\mu(xH)\\
&=& \int_{G/H} \int_H \varphi(xH)\,g(x\xi)\,d\lambda_H(\xi)\,d\mu(xH)\\
&=& \int_{G/H} \int_H g(x\xi)\,d\lambda_H(\xi)\,d\mu_\varphi(xH)\\
&=& m_{\mu_\varphi}(g),
\end{eqnarray*}
where $g\in C_0(G)$. Thus, by considering the canonical embeddings $L^1(G)\rightarrow M(G)$ and $L^1(G/H)\rightarrow M(G/H)$ we have the following commutative diagram:
\[
  \begin{CD}
  L^1(G:H) @> T >>  L^1(G/H)\\
    @ V  VV  @VV  V \\
   M(G:H) @>> \tilde{T}  >  M(G/H)
  \end{CD}
  \]

\end{exam}


\section{$M(G/H)$ as a Closed Subalgebra of $M(G)$}

In this section, by assuming the compactness of $H$, we define a convolution on $M(G/H)$ to give it the construction of a Banach algebra.
In fact, we will show that $M(G:H)$, the Banach space identified with $M(G/H)$, is a subalgebra of $M(G)$ and so the convolution of $M(G:H)$ induces a convolution on $M(G/H)$. In addition, the obtained Banach algebra $M(G/H)$ has no identity and no involution, where $H$ is a subgroup of $G$ which is not normal.

\begin{lem}\label{lma04}
For all $\varphi\in C_0(G/H)$ and $\nu\in M(G/H)$, the functions which are defined as follow are continuous functions vanishing at infinity:
\begin{eqnarray*}
\left\{\begin{array}{clcr}
\psi:G/H\rightarrow \mathbb{C}\\
yH\mapsto\int_{G/H} \varphi_{yH}(xH)\,d\nu(xH)
\end{array}
\right.
\quad\text{ and}\quad
\left\{\begin{array}{clcr}
\eta:G/H\rightarrow \mathbb{C}\\
yH\mapsto\int_{G/H} \leftidx{_{yH}}\varphi(xH)\,d\nu(xH)
\end{array}
\right.
\end{eqnarray*}
\end{lem}

\begin{proof}
Obviously, the mappings $yH\mapsto \leftidx{_{yH}}\varphi$ and $yH\mapsto \psi(yH)$ are continuous.
Take a compact subset $K$ of $G$ for which $|\varphi(yH)|<\varepsilon$ where $yH\notin q(K)$. Obviously, $E=q(Hx^{-1}K)$ is compact and for all $\xi\in H$,
\begin{eqnarray*}
x\xi yH\in q(K) \Leftrightarrow x\xi y\in KH \Leftrightarrow y\in \xi^{-1}x^{-1}KH \Rightarrow y\in Hx^{-1}KH \Leftrightarrow yH\in q(Hx^{-1}K)=E.
\end{eqnarray*}
Thus, if $yH\notin E$, then $x\xi yH\notin q(K)$ and so $|\psi(yH)| <\varepsilon\,\|\nu\|.$
Therefore, $\psi\in C_0(G/H)$.

\end{proof}

\bigskip

For all $\nu,\,\omega\in M(G/H)$, define
\begin{eqnarray}
\left\{\begin{array}{clcr}
S:C_0(G/H)\rightarrow \mathbb{C}\\
\varphi\rightarrow \int_{G/H} \int_{G/H} \varphi_{yH}(xH)\,d\nu(xH)\,d\omega(yH)
\end{array}
\right.
\end{eqnarray}
Then $S$ is a bounded linear functional with $\|S\|\leq\|\nu\|\,\|\omega\|$. Since
\begin{eqnarray*}
|S(\varphi)| &=& |\int_{G/H} \int_{G/H} \int_H \varphi(x\xi yH)\,d\lambda_H(\xi)\,d\nu(xH)\,d\omega(yH)|\\
&\leq & \int_{G/H} \int_{G/H} \int_H |\varphi(x\xi yH)|\,d\lambda_H(\xi)\,d|\nu|(xH)\,d|\omega|(yH)\\
&\leq& \|\varphi\|_{\sup}\,\|\nu\|\,\|\omega\|
\end{eqnarray*}
for all $\varphi\in C_0(G/H)$. Hence, $S$ determines an element $\nu *\omega\in M(G/H)$, called the {\it convolution} of $\nu$ and $\omega$, such that for all $\varphi\in C_0(G/H)$, $S(\varphi)=(\nu*\omega)(\varphi)$, i.e.;
\begin{eqnarray}\label{007}
\int_{G/H} \varphi(xH)\,d(\nu*\omega)(xH) &=& \int_{G/H} \int_{G/H} \int_H \varphi(x\xi yH)\,d\lambda_H(\xi)\,d\nu(xH)\,d\omega(yH)\\
&=& \int_{G/H} \int_{G/H} \varphi_{yH}(xH)\,d\nu(xH)\,d\omega(yH)\nonumber\\
&=& \int_{G/H} \int_{G/H} \leftidx{_{xH}}\varphi(yH)\,d\nu(xH)\,d\omega(yH).\nonumber
\end{eqnarray}

We should remark that the convolution of two measures has a simpler form if one of them is a dirac measure. In details, one may easily check that
\newcounter{mh7}
\begin{list}
{\bf(\alph{mh7})}{\usecounter{mh7}}
\item $(\nu*\delta_{xH})(\varphi)=\int_{G/H} \varphi_{xH}(yH)\,d\nu(yH)$,
\item $(\delta_{xH}*\nu)(\varphi)=\int_{G/H} \leftidx{_{xH}}\varphi(yH)\,d\nu(yH)$,
\item
\begin{eqnarray}
\delta_{xH}*\delta_{yH}=\int_H\delta_{x\xi yH}\,d\lambda_H(\xi)\quad(x,y\in G)
\end{eqnarray}
where the integral is considered as a vector-valued integral. That is
\begin{eqnarray}
(\delta_{xH}*\delta_{yH})(\varphi)=\int_H\varphi(x\xi yH)\,d\lambda_H(\xi)=\leftidx{_{xH}}\varphi(yH)=\varphi_{yH}(xH),
\end{eqnarray}
\end{list}
where $x,y\in G$, $\nu\in M(G/H)$, and $\varphi\in C_0(G/H)$.\\

As $M(G:H)$ is a Banach algebra, the isometry isomorphism of Banach spaces\break $\tilde{T}:M(G:H)\rightarrow M(G/H)$ makes $M(G/H)$ into a Banach algebra. The following lemma shows that the convolution on $M(G/H)$ is the convolution induced by the convolution on $M(G:H)$.

\begin{lem}\label{lma03}
For all $\nu,\omega\in M(G/H)$,
\newcounter{mh6}
\begin{list}
{\bf(\alph{mh6})}{\usecounter{mh6}}
\item $\nu*\omega=\tilde{T}(m_\nu*m_0)$ where $m_0\in M(G)$ satisfies $\tilde{T}(m_0)=\omega$. Specially,\break $\nu*\omega=\tilde{T}(m_\nu*m_\omega)$.
\item $m_{\nu*\omega}=m_\nu *m_\omega$.
\end{list}
\end{lem}

\begin{proof}
{\bf (a) } Suppose that $m_0\in M(G)$ and $\tilde{T}(m_0)=\omega$. Then for all $\varphi\in C_0(G/H)$ we have
\begin{eqnarray*}
\tilde{T}(m_\nu*m_0)(\varphi) &=& \int_G (\varphi oq)(x)\,d(m_\nu*m_0)(x)\\
&=& \int_G \int_G R_y(\varphi oq)(x)\,dm_\nu(x)\,dm_0(y)\\
&=& \int_G \int_{G/H}\int_H R_y(\varphi oq)(x\xi)\,d\lambda_H(\xi)\,d\nu(xH)\,dm_0(y)\\
&=& \int_G \int_{G/H}\int_H \varphi(x\xi yH)\,d\lambda_H(\xi)\,d\nu(xH)\,dm_0(y)\\
&=& \int_G \int_{G/H} \varphi_{yH}(xH)\,d\nu(xH)\,dm_0(y)\\
&=& \int_G \psi(yH)\,dm_0(y),
\end{eqnarray*}
where $\psi\in C_0(G/H)$ is defined by $\psi(yH)=\int_{G/H} \varphi_{yH}(xH)\,d\nu(xH)$. Therefore,
\begin{eqnarray*}
\tilde{T}(m_\nu*m_0)(\varphi) &=& \int_{G/H} \psi(yH)\,d(\tilde{T}m_0)(yH)\\
&=& \int_{G/H} \int_{G/H} \varphi_{yH}(xH)\,d\nu(xH)\,d\omega(yH)\\
&=& \int_{G/H} \varphi\,d(\nu*\omega)(xH).
\end{eqnarray*}
\medskip
{\bf (b)} For all $f\in C_0(G)$ we can write
\begin{eqnarray*}
\int_G f(x)\,dm_{\nu*\omega}(x) &=& \int_{G/H}T_\infty f(xH)\,d(\nu*\omega)(xH)\\
&=& \int_{G/H}\int_{G/H}(T_\infty f)_{yH}(xH)\,d\nu(xH)\,d\omega(yH)\\
&=& \int_{G/H}\int_{G/H}\int_H (T_\infty f)(x\xi yH)\,d\lambda_H(\xi)\,d\nu(xH)\,d\omega(yH)\\
&=& \int_{G/H}\int_{G/H}\int_H \int_H f(x\xi y\eta)\,d\eta\,d\lambda_H(\xi)\,d\nu(xH)\,d\omega(yH)\\
&=& \int_{G/H}\int_H \big(\int_H \int_{G/H}R_{y\eta}f(x\xi)\,d\lambda_H(\xi)\,d\nu(xH)\big )\,d\eta\,d\omega(yH)\\
&=& \int_G \int_H \int_{G/H}R_{y}f(x\xi)\,d\lambda_H(\xi)\,d\nu(xH)\,dm_\omega(y)\\
&=& \int_G \int_G R_{y}f(x)\,dm_\nu(x)\,dm_\omega(y)\\
&=& \int_G f(x)\,d(m_\nu*m_\omega)(x)
\end{eqnarray*}
\end{proof}

\begin{thm}
Let $H$ be a compact subgroup of $G$. Then $M(G/H)$  with the convolution defined by \eqref{007} is a Banach algebra.
\end{thm}

\begin{proof}
For all $\nu,\omega,\kappa\in M(G/H)$, by using Lemma \ref{lma03}, we have
\begin{eqnarray*}
(\nu*\omega)*\kappa &=&  \tilde{T}(m_{(\nu*\omega)}*m_\kappa) = \tilde{T}\big((m_\nu*m_\omega)*m_\kappa\big) \\
&=& \tilde{T}\big( m_\nu*(m_\omega*m_\kappa)\big) =\tilde{T}( m_\nu*m_{\omega*\kappa}) = \nu*(\omega*\kappa).
\end{eqnarray*}

\end{proof}

We should notice that if $H$ is a normal subgroup of $G$, then $G/H$ is a locally compact topological group and the convolution on $M(G/H)$ defined by \eqref{007} coincides with the well-known convolution on the measure algebra of $G/H$. In this case, $M(G/H)$ is a unital involutive Banach algebra whose identity is $\delta_{eH}$. \\
When $H$ is a compact subgroup of $G$, the Banach algebra $M(G/H)$ has a right identity $\delta_{eH}$. To show this, let $\nu\in M(G/H)$. Then for all $\varphi\in C_0(G/H)$,
\begin{eqnarray*}
(\nu*\delta_{eH})(\varphi) &=& \int_{G/H}\int_{G/H}\int_H \varphi(x\xi yH)\,d\lambda_H(\xi)\,d\nu(xH)\,d\delta_{eH}(yH)\\
&=& \int_{G/H}\int_H \varphi(x\xi eH)\,d\lambda_H(\xi)\,d\nu(xH)\\
&=& \int_{G/H}\int_H \varphi(xH)\,d\lambda_H(\xi)\,d\nu(xH)\\
&=& \nu(\varphi).
\end{eqnarray*}
The following theorem indicates that a necessary and sufficient condition to have an identity in $M(G/H)$ is that $H$ is a normal subgroup.

\begin{prop}\label{unital normal}
In Banach algebra $M(G/H)$, $\delta_{eH}$ satisfies $\delta_{eH}*\mu_\varphi=\mu_\varphi$ for all $\varphi\in L^1(G/H)$ if and only if $H$ is a normal subgroup of $G$.
\end{prop}

\begin{proof}
Suppose that $H\ntrianglelefteq G$. Then there exists some $x\in G$ and $\xi_0\in H$ for which\break $\xi_0xH\neq xH$. By taking a function $\varphi\in C_c^+(G/H)$ with $\varphi(xH)=0$ and $\varphi(\xi_0xH)>0$, we get a non-negative and non-zero continuous function $H\rightarrow \mathbb{C}$, $\xi\mapsto \varphi(\xi xH)$, which implies that $\leftidx{_{eH}}\varphi = \int_H \varphi(\xi xH)\,d\lambda_H(\xi) >0$. Therefore,
\begin{eqnarray*}
0=\varphi(xH) \neq \leftidx{_{xH}}\varphi(xH) = \int_H \varphi(\xi xH)\,d\lambda_H(\xi) >0
\end{eqnarray*}
On the other hand, $\delta_{eH}* \mu_\varphi=\leftidx{_{xH}}\varphi$. So $\delta_{xH}*\mu_\varphi\neq\mu_\varphi$
\end{proof}

\begin{prop}
When $H$ is a compact subgroup of $G$, the following statements are equivalent:
\newcounter{mh9}
\begin{list}
{\bf(\alph{mh9})}{\usecounter{mh9}}
\item $H$ is a normal subgroup of $G$.
\item $M(G/H)$ has an identity.
\item $M(G/H)$ has an approximate identity.
\end{list}
\end{prop}

\begin{proof}
We need to prove $(c)\Rightarrow (a)$. If $M(G/H)$ has an approximate identity $\{\nu_\alpha\}$, then for all $\mu\in M(G/H)$ we have
\begin{eqnarray*}
\delta_{eH}*\mu &=& \lim_\alpha(\nu_\alpha*\delta_{eH})*\mu\\
&=&  \lim_\alpha\nu_\alpha*\mu\\
&=& \mu.
\end{eqnarray*}
This together with Proposition~\ref{unital normal} imply that $H$ must be normal.
\end{proof}

\begin{cor}
If $H$ is a compact subgroup of $G$, then $M(G/H)$ is amenable if and only if $H$ is a normal open subgroup of $G$ and $G$, or equivalently $G/H$, is amenable.
\end{cor}

We should mention that if $A$ is an involutive Banach algebra with a right (left) identity, then it has an identity. Since if $e$ is a right identity, for all $a\in A$ we have
\begin{eqnarray*}
e^**a = (a^* *e)^*=a
\end{eqnarray*}
which shows that $e^*$ is a left identity and hence $e=e^*$ is the identity of $A$.

\begin{cor}
There is an involution on Banach algebra $M(G/H)$ only if $H$ is a normal subgroup of $G$.
\end{cor}


\section{Banach Subalgebra $L^1(G/H)$ of $M(G/H)$}

When $H$ is a compact subgroup of $G$ and $G/H$ has been attached to a strongly quasi-invariant measure $\mu$, then for all $\varphi\in L^1(G/H)$, the function defined almost everywhere by\break $\varphi_\rho(x)= \varphi oq(xH)\,\rho(x)$ belongs to $L^1(G)$, $\varphi=Tf$, and $\|\varphi_\rho\|_1=\|\varphi\|_1$.

\begin{thm}
When $G/H$ has been attached to a strongly quasi-invariant measure $\mu$, $L^1(G/H)$ is an ideal of $M(G/H)$. More precisely, for all $\varphi\in L^1(G/H)$ and $\nu\in M(G/H)$,
\begin{eqnarray}\label{mh03}
(\varphi*\nu)(xH)= \int_{G/H} \Delta_G(y^{-1})\int_H\,\varphi(x\xi y^{-1}H)\,\frac{\rho(x\xi y^{-1})}{\rho(x)}\,d\lambda_H(\xi)\,d\nu(yH)
\end{eqnarray}
and
\begin{eqnarray}
(\nu*\varphi)(xH)= \int_{G/H}\int_H\varphi(\xi y^{-1}xH)\,\lambda(\xi y^{-1},xH)\,d\lambda_H(\xi)\,d\nu(yH)
\end{eqnarray}
for $\mu$-almost all $yH\in G/H$.
\end{thm}

\begin{proof}
First suppose that $\nu,\omega\in M(G/H)$ and $\omega\ll\mu$. Then by taking $\varphi\in L^1(G/H)$ with $\omega=\mu_\varphi$, we have $m_\nu *\lambda_{\varphi_\rho}\ll \lambda$ and so $m_\nu *\lambda_{\varphi_\rho}= \lambda_g$ for some $g\in L^1(G)$.
\begin{eqnarray*}
\nu*\omega=\nu*  \mu_\varphi= \tilde{T}(m_\nu *\lambda_{\varphi_\rho})= \tilde{T}(\lambda_g) =\mu_{Tg}\ll\mu
\end{eqnarray*}
Also, there exists some $f\in L^1(G)$ such that $\varphi=\mu_{Tf}$. Obviously, $m_\nu *\lambda_f\ll \lambda$ and so $m_\nu *\lambda_f= \lambda_g$ for some $g\in L^1(G)$. Therefore,
\begin{eqnarray*}
\nu*\omega=\nu *  \tilde{T}\lambda_f= \tilde{T}(m_\nu *\lambda_f)= \tilde{T}(\lambda_g) =\mu_{Tg}\ll\mu.
\end{eqnarray*}

Suppose that $\varphi\in L^1(G/H)$ and $\nu\in M(G/H)$. Then for all $\psi\in C_0(G/H)$,
\begin{eqnarray*}
(\mu_\varphi *\nu)(\psi) &=& \int_{G/H} \int_{G/H} \psi_{yH}(xH)\,\varphi(xH)\,d\mu(xH)\,d\nu(yH)\\
&=& \int_{G/H} \langle\mathcal{R}_y\psi,\varphi\rangle\,d\nu(yH)\\
&=& \int_{G/H} \Delta_G(y^{-1})\,\langle\psi,\mathcal{R}_{y{-1}}\varphi\rangle\,d\nu(yH)\\
&=& \int_{G/H} \psi(xH)\,\int_{G/H} \Delta_G(y^{-1})\int_H\,\varphi(x\xi y^{-1}H)\,\frac{\rho(x\xi y^{-1})}{\rho(x)}\,d\lambda_H(\xi)\,d\nu(yH)\,d\mu(xH)\\
&=& \int_{G/H} \psi(xH)\,\gamma(xH)\,d\mu(xH)\\
&=& \mu_\gamma(\psi),
\end{eqnarray*}
where $\gamma(xH)= \int_{G/H} \Delta_G(y^{-1})\int_H\,\varphi(x\xi y^{-1}H)\,d\lambda_H(\xi)\,d\nu(xH)$, for $\mu$-almost all $xH\in G/H$. Also,
\begin{eqnarray*}
(\nu*\mu_\varphi)(\psi) &=& \int_{G/H} \int_{G/H} \int_H \psi(x\xi yH)\,d\lambda_H(\xi)\,d\nu(xH)\,\varphi(yH)\,d\mu(yH)\\
&=&  \int_{G/H} \int_H \int_{G/H}\psi(yH)\,\varphi(\xi^{-1}x^{-1}yH)\,\lambda(\xi^{-1}x^{-1},yH)\,d\mu(yH)\,d\lambda_H(\xi)\,d\nu(xH)\\
&=&  \int_{G/H} \psi(yH)\,\int_{G/H}\int_H\varphi(\xi x^{-1}yH)\,\lambda(\xi x^{-1},yH)\,d\lambda_H(\xi)\,d\nu(xH)\,d\mu(yH)\\
&=& \int_{G/H} \psi(yH)\,\zeta(yH)\,d\mu(yH)\\
&=& \mu_\zeta(\psi),
\end{eqnarray*}
where $\zeta(yH)= \int_{G/H}\int_H\varphi(\xi x^{-1}yH)\,\lambda(\xi x^{-1},yH)\,d\lambda_H(\xi)\,d\nu(xH)$, for $\mu$-almost all $yH\in G/H$.
\end{proof}

By considering $L^1(G/H)$ as a subalgebra of $M(G/H)$, where $\mu$ is the strongly quasi-invariant measure arising from a rho-function $\rho$, we may define a convolution on $L^1(G/H)$ by using the restriction the convolution on $M(G/H)$. It means that, for all $\varphi,\psi\in L^1(G/H)$, $\varphi*\psi$ is the element of $L^1(G/H)$ for which $\mu_\varphi*\mu_\psi=\mu_{\varphi*\psi}$. The following theorem states that how we can obtain the convolution on $L^1(G/H)$ by using the convolution on $L^1(G)$ and it also shows that we can assume $\varphi*\psi$ as a generalized linear combination of the left translations of $\psi$.\\

\begin{thm}
Suppose that $H$ is a compact subgroup of $G$ and $\mu$ is the strongly quasi-invariant Radon
measure on $G/H$ arising from a rho-function $\rho$. For all $\varphi,\psi\in L^1(G/H)$ define $\varphi*\psi$, the
convolution of $\varphi$ and $\psi$, by
\begin{eqnarray}
\varphi*\psi&=&\int_G
\varphi_\rho(y)\,\mathcal{L}_y\psi\,d\lambda_G(y)\\
&=& T(\varphi_\rho*\psi_\rho)
\quad(\mu-{\rm almost\;\,all\;\,}xH\in
\frac{G}{H})
\end{eqnarray}
where the integrals are considered in the sense of vector-valued integrals. Then, $L^1(G/H)$ is a Banach algebra.
\end{thm}

\begin{proof}
Suppose that $\varphi,\psi\in L^1(G/H)$. By using \eqref{mh03}, for $\mu$-almost all $xH\in G/H$ we have
\begin{eqnarray*}\label{mh05}
T(\varphi_\rho *\psi_\rho)(xH)&=& \int_H \frac{(\varphi_\rho*\psi_\rho)}{\rho(x\xi)}\,d\lambda_H(\xi)\\
&=& \int_H \int_G \frac{\varphi(yH)\,\rho(y)\,\psi(y^{-1}x\xi H)\,\rho(y^{-1}x\xi)}{\rho(x)}\,d\lambda_G(y)\,d\lambda_H(\xi)\\
&=& \int_G \varphi(yH)\,\rho(y)\, \int_H\frac{\psi(y^{-1}x H)\,\rho(y^{-1}x)}{\rho(x)}\,d\lambda_H(\xi)\,d\lambda_G(y)\\
&=& \int_G \varphi(yH)\,\rho(y)\,\psi(y^{-1}x H)\,\lambda(y^{-1},xH)\,d\lambda_G(y)\\
&=& \int_G \varphi(yH)\,\rho(y)\,\mathcal{L}_y\psi(x H)\,d\lambda_G(y)\\
&=& \int_G \varphi(xyH)\,\rho(xy)\,\mathcal{L}_{xy}\psi(x H)\,d\lambda_G(y)\\
&=& \int_G \varphi(xyH)\,\rho(xy)\,\lambda(y^{-1}x^{-1},xH)\,\psi(y^{-1} H)\,d\lambda_G(y)\\
&=& \frac{1}{\rho(x)}\,\int_G \varphi(xyH)\,\rho(xy)\,\rho(y^{-1})\,\psi(y^{-1} H)\,d\lambda_G(y)\\
&=& \frac{1}{\rho(x)}\,\int_{G} \varphi(xy^{-1}H)\,\rho(xy^{-1})\,\rho(y)\,\psi(y H)\,\Delta_G(y^{-1})\,d\lambda_G(y)\\
&=& \frac{1}{\rho(x)}\,\int_{G/H}\Delta_G(y^{-1})\,\int_H \varphi(x\xi^{-1} y^{-1}H)\,\rho(x\xi^{-1} y^{-1})\,\psi(y\xi H)\,d\lambda_G(y)\,d\mu(yH)\\
&=& \frac{1}{\rho(x)}\,\int_{G/H}\Delta_G(y^{-1})\,\int_H \varphi(x\xi^{-1} y^{-1}H)\,\rho(x\xi^{-1} y^{-1})\,\psi(yH)\,d\lambda_G(y)\,d\mu(yH)\\
&=& \frac{1}{\rho(x)}\,\int_{G/H}\Delta_G(y^{-1})\,\int_H \varphi(x\xi^{-1} y^{-1}H)\,\rho(x\xi^{-1} y^{-1})\,d\lambda_G(y)\,d\mu_\psi(yH)\\
&=& (\varphi*\mu_\psi)(xH)\\
&=& (\varphi*\psi)(xH),
\end{eqnarray*}
for $\mu$-almost all $xH\in G/H$. The above equations show that
\begin{eqnarray*}
\varphi*\psi(xH)&=&\int_G
\varphi(yH)\,\rho(y)\,\mathcal{L}_y\psi(xH)\,d\lambda_G(y)\\
&=&\int_G
\varphi_\rho(y)\,\mathcal{L}_y\psi(xH)\,d\lambda_G(y)\\
&=& T(\varphi_\rho*\psi_\rho)
\end{eqnarray*}
for $\mu$-almost all $xH\in G/H$.
\end{proof}

In other words, $L^1(G:H)$ is a left ideal and a closed subalgebra of $L^1(G)$.\\

\begin{thm}
Let $H$ be a compact subgroup of $G$ and $G/H$ have been attached to a strongly quasi-invariant measure $\mu$. Then
\newcounter{mh5}
\begin{list}
{\bf(\alph{mh5})}{\usecounter{mh5}}
\item $L^1(G/H) = \{\nu\in M(G/H):\, xH\mapsto |\nu|*\delta_{xH} \text{ is weakly continuous}\}$

\item  $L^1(G/H) \subseteq\{\nu\in M(G/H):\, xH\mapsto \delta_{xH}*|\nu| \text{ is weakly continuous}\}$
\end{list}
\end{thm}

\begin{proof}
First, suppose that $d\nu(xH)=\varphi(xH)\,d\mu(xH)$ for some $\varphi\in L^1(G/H)$. Then for all $x,y\in G$ and $\psi\in C_0(G/H)$,
\begin{eqnarray*}
|(|\nu|*\delta_{xH})(\psi)-|\nu|*\delta_{yH})(\psi)| &=& |\int_{G/H} \psi_{xH}(zH)-\psi_{yH}(zH)\,d|\nu|(zH)|\\
&\leq& \int_{G/H} |\mathcal{R}_x\psi(zH)-\mathcal{R}_y\psi(zH)|\,|\varphi(zH)\,d\mu(zH)|\\
&\leq&  \|\mathcal{R}_x\psi-\mathcal{R}_y\psi\|_{\sup}\,\|\varphi\|_1.
\end{eqnarray*}
Since $\psi\in C_0(G/H)\subseteq RUC(G/H)$, $\|\mathcal{R}_x\psi-\mathcal{R}_y\psi\|_{\sup}\rightarrow 0$ as $x\rightarrow y$. Thus,\break $|(|\nu|*\delta_{xH})(\psi)-|\nu|*\delta_{yH})(\psi)|\rightarrow 0$ as $x\rightarrow y$. Similarly, $\psi$ is left uniformly continuous and
\begin{eqnarray*}
|(\delta_{xH}*|\nu|)(\psi)-\delta_{yH}*|\nu|)(\psi)| &=& |\int_{G/H} \leftidx{_{xH}}\psi(zH)-\leftidx{_{yH}}\psi(zH)\,d|\nu|(zH)|\\
&\leq& \int_{G/H} |\mathcal{L}_x\psi(zH)-\mathcal{L}_y\psi(zH)|\,|\varphi(zH)\,d\mu(zH)|\\
&\leq&  \|\mathcal{L}_x\psi-\mathcal{L}_y\psi\|_{\sup}\,\|\varphi\|_1\\
&\longrightarrow & 0
\end{eqnarray*}
as $x\rightarrow y$.\\
Now, let $\nu\in M(G/H)$ be such that $xH\mapsto |\nu|*\delta_{xH}$ is weakly continuous. Since $L^1(G/H)$ is an ideal of $M(G/H)$, when $\varphi \in L^1(G/H)$, $|\nu|*\mu_\varphi\in L^1(G/H)$. In other words, $|\nu|*\mu_\varphi\ll \mu$. To show that $\nu\ll \mu$, assume that $E\in \mathcal{B}(G/H)$ is null with respect to $\mu$. By taking a neighborhood $U$ of $e$ with compact closure and $\varphi=\chi_{q(U)}$, we can write
\begin{eqnarray*}
0 = (|\nu|*\mu_\varphi)(E) &=& \int_{G/H}\int_{G/H} (\chi_E)_{yH}(xH)\varphi(yH),d|\nu|(xH)\,d\mu(yH)\\
&=& \int_{q(U)}\int_{G/H} (\chi_E)_{yH}(xH)\,d|\nu|(xH)\,d\mu(yH)\\
&=& \int_{q(U)} (|\nu|*\delta_{yH})(\chi_E)\,d\mu(yH).
\end{eqnarray*}
Since, $yH\mapsto (|\nu|*\delta_{yH})(\chi_E)$ is a non-negative and continuous function, we deduce that
\begin{eqnarray*}
(|\nu|*\delta_{yH})(\chi_E)=0\quad(yH\in q(U)).
\end{eqnarray*}
In particular, by taking $yH=eH$ we get $|\nu|(E)=0$ and so $\nu(E)=0$.

\end{proof}

\begin{thm}
$L^1(G/H)$ has a left approximate identity if and only if $H$ is a normal subgroup of $G$.
\end{thm}

\begin{proof}
Let $\{\eta_\alpha\}_{\alpha\in I}$ be a left approximate identity for $L^1(G/H)$. Then for all $\varphi\in L^1(G/H)$ we have $\eta_\alpha*\varphi\longrightarrow\varphi$ and so $\mu_{\eta_\alpha}*\mu_{\varphi}\longrightarrow\mu_{\varphi}$. Therefore,
\begin{eqnarray*}
\mu_\varphi = \lim_\alpha \mu_{\eta_\alpha}*\mu_{\varphi}= \lim_\alpha (\mu_{\eta_\alpha}*\delta_{eH})*\mu_{\varphi}= \lim_\alpha \mu_{\eta_\alpha}*(\delta_{eH}*\mu_{\varphi})= \delta_{eH}*\mu_{\varphi}
\end{eqnarray*}
where $\varphi\in L^1(G/H)$, and hence $H$ must be normal in $G$.
\end{proof}

If $A$ is an involutive Banach algebra with a right approximate identity $\{e_\alpha\}_{\alpha\in I}$, then it has a left approximate identity. Since, for all $a\in A$ we have $\lim e_\alpha^**a = (\lim a^**e_\alpha)^*=a^{**}=a$.

\begin{cor}
There is an involution on Banach algebra $L^1(G/H)$ only if $H$ is a normal subgroup of $G$.
\end{cor}

\begin{cor}
If $H$ is a compact subgroup of $G$ and $G/H$ has been attached to a strongly quasi-invariant measure, then $L^1(G/H)$ is amenable if and only if $H$ is a normal subgroup of $G$ and $G$, or equivalently $G/H$, is amenable.
\end{cor}

\begin{cor}
Let $H$ be a compact subgroup of $G$ and $G/H$ have been attached to a strongly quasi-invariant measure $\mu$. Then
$L^1(G/H) =\{\nu\in M(G/H):\, xH\mapsto \delta_{xH}*|\nu| \text{ is weakly continuous}\}$
just when $H$ is a normal subgroup of $G$.
\end{cor}



\medskip
{\bf Hossein Javanshiri}\\
Department of Mathematics, Yazd University, Yazd, Iran.\\
h.javanshiri@yazd.ac.ir\\
\medskip
{\bf Narguess Tavallaei}\\
Department of Mathematics, School of Mathematics and Computer Science,\\
Damghan University, Damghan, Iran.\\
tavallaie@du.ac.ir

\end{document}